\documentclass[11pt]{article}

\usepackage{color}

\newcommand{\R}{\mathbb R}

\newcommand{\T}{\mathbb T}
\newcommand{\N}{\mathbb N}

\usepackage{amsmath, amsfonts, amssymb}
\usepackage{graphicx, amsthm, color}
\usepackage{cite}

\newtheorem{theorem}{Theorem}[section]
\newtheorem{lemma}[theorem]{Lemma}

\theoremstyle{definition}

\theoremstyle{remark}

\numberwithin{equation}{section}
\begin{document}

\begin{center}
  {\LARGE A convex function satisfying the {\L}ojasiewicz inequality\\[2mm] but failing the gradient conjecture
    both at zero and infinity.}
\end{center}

\smallskip

\begin{center}
\textsc{Aris Daniilidis, Mounir Haddou, Olivier Ley}
\end{center}

\medskip

\noindent\textbf{Abstract.}
We construct an example of a smooth convex function on the plane
with a strict minimum at zero, which is real analytic except at zero,
for which Thom's gradient conjecture fails both
at zero and infinity. More precisely, the gradient orbits of the
function spiral around zero and at infinity. 
Besides, the function satisfies the {\L}ojasiewicz gradient inequality
at zero.

\bigskip


\vspace{0.6cm}

\textit{Secondary} 34A26, 34C08, 52A41.


\section{Introduction}

Answering a question of Whitney, {\L}ojasiewicz~\cite{lojasiewicz63}
showed that every analytic variety $f^{-1}(0)$, where $f:\mathcal{U}\subset \mathbb{R}^N\rightarrow\mathbb{R}$ is real-analytic ($\mathcal{U}\neq\emptyset$, open),
is a deformation retract of its open neighborhood. The
deformation was given by the flow of the Euclidean gradient $-\nabla (f^2)$.
The main argument of {\L}ojasiewicz was based on a famous lemma, nowadays known as
the {\L}ojasiewicz (gradient) inequality, which asserts that for some
$\vartheta\in(0,1)$ and $c>0$ we have
\begin{equation}
\Vert\nabla f(x)\Vert\geq c|f(x)-f(a)|^{\vartheta}\label{eq:1}
\end{equation}
for all $x$ sufficiently close to $a\in f^{-1}(0)$. The above inequality ensures that every bounded gradient orbit $t\mapsto \gamma (t)$
({\em i.e.}, $\dot \gamma=\nabla f(\gamma)$) has finite length and therefore converges to a singular point $\gamma_{\infty}$
with $\nabla f(\gamma_{\infty})=0$. \smallskip

Some years later, Thom conjectured that in this case,
up to a change of coordinates that identifies $\gamma_{\infty}$ to~$0$, the
spherical part of the orbit also converges. In other words, the limit of secants
\begin{equation}
\lim_{t\rightarrow+\infty}\frac{\gamma(t)-\gamma_{\infty}}{||\gamma(t)-\gamma_{\infty}
||}\text{\quad exists.}\label{eq:thom}
\end{equation}
For decades, this has been known as the (Thom) gradient
conjecture, see \cite{arnold83, thom89}.  (For the more general problem of non-oscillation of trajectories, we refer to~\cite{moussu97,cms07,gs13}.)
The gradient conjecture makes sense for any gradient dynamics for which bounded orbits
converge. Partial results revealed that~\eqref{eq:thom}
should hold in the real-analytic case, see ~\cite{ichikawa92,
lin92, sanz98}, fact that was eventually published in
full generality by Kurdyka, Mostowski and Parusi\'{n}ski~\cite{kmp00} in 2000.
The proof was based on \eqref{eq:1} together with concrete analytic estimations.\smallskip

{\L}ojasiewicz showed that the gradient inequality \eqref{eq:1} remains valid also for
$\mathcal{C}^{1}$ semialgebraic (respectively, globally subabalytic)
functions, see \cite{lojasiewicz84}. In 1998, Kurdyka~\cite{kurdyka98} generalized
\eqref{eq:1} for $\mathcal{C}^{1}$ functions that are \textit{definable} in
some \textit{o-minimal structure}, an axiomatic definition due to van den
Dries~\cite{dm96,dries98} which encompasses
semialgebraic and globally subanalytic functions, but also larger
classes that include the exponential function~\cite{miller94}. More precisely, Kurdyka showed
that for every definable function $f$ and critical value
$r_{\infty}$ (which is necessarily isolated) there exists $\delta>0$
and a continuous function $\Psi:[r_{\infty},r_{\infty}+\delta)\rightarrow
\mathbb{R}$ which is $\mathcal{C}^{1}$ on $(r_{\infty},r_{\infty}+\delta)$
with $\Psi^{\prime}>0$ such that
\begin{equation}
||\nabla(\Psi\circ f)(x)||\geq1 \label{eq:2}
\end{equation}
for all $x\in\mathbb{R}^{N}$ such that $r_{\infty}<f(x)<r_{\infty}+\delta$. In
addition, Kurdyka's proof showed that the function $\Psi$ can be taken in the same o-minimal
structure as $f$. Consequently, if $f$ is semialgebraic or globally
subanalytic, then so is $\Psi$ and thanks to Puiseux's theorem we may take
$\Psi(r)=r^{1-\vartheta}$, for $\vartheta\in(0,1)$. It is then straightforward
to see that (\ref{eq:2}) actually yields (\ref{eq:1}) for $c=(1-\vartheta
)^{-1}.$\smallskip

We refer to \eqref{eq:2} as the Kurdyka-{\L}ojasiewicz (in short, K{\L}) inequality and we call 
K{\L}-function any function with (upper) isolated critical values that satisfies the
K{\L}-inequality around any of them. Similarly to the gradient inequality
\eqref{eq:1}, bounded gradient orbits of a K{\L}-function have finite
length. There are well-known examples of $\mathcal{C}^{\infty}$ functions in
$\mathbb{R}^{2}$ with isolated critical values that are not K{\L}-functions
(they have bounded gradient orbits which fail to converge), see \cite{fokin81, pd82}. Bounded gradient orbits of convex functions have finite length~\cite{dddl15, mp91} and therefore converge, but there are also examples 
of $\mathcal{C}^{2}$-smooth convex functions failing K{\L}-property, see \cite[\S 4.3]{bdlm10} or
\cite[\S 5.1]{bp20}. In \cite{bdlm10} we characterized the class of K{\L}-functions (among the ones with upper isolated critical values) and gave criteria for a convex function to be~K{\L}.\smallskip

In \cite{kp06}, Kurdyka and Parusinski used K{\L}-inequality together with a quasiconvex cell decomposition of o-minimal sets and concrete estimates to show that the gradient conjecture holds for $\mathcal{C}^{1}$ o-minimal functions 
provided either $N=2$ (planar case) or the structure is \textit{polynomially bounded} (in particular if $f$ is semialgebraic or globally subanalytic). On the other hand, mere convexity is not sufficient to guarantee~\eqref{eq:thom}: there exist examples of convex functions whose orbits either spiral \cite[\S 7.2]{dls10} or oscillate between two secants~\cite{bp20}.\smallskip

In \cite{grandjean07}, Grandjean considered the behavior of the secants
at infinity: he showed that if $f$ is a $\mathcal{C}^{1}$
semialgebraic function and $t\mapsto \gamma(t)$ is a gradient orbit satisfying
$||\gamma(t)||\rightarrow\infty,$ as $t\rightarrow+\infty,$ then the limit of
secants at infinity
\begin{equation}
\lim_{t\rightarrow+\infty}\frac{\gamma(t)}{||\gamma(t)||}\text{\quad exists\qquad
(gradient conjecture at infinity).}\label{eq:thom-inf}
\end{equation}
The proof is based on a {\L}ojasiewicz type gradient inequality at infinity
previously obtained by the author together with D'Acunto in~\cite{dg05}.\smallskip

The behavior of secants at infinity has recently become relevant in Machine Learning. If a deep network model is unbiased and homogeneous (max-pooling, ReLu, linear and convolutional layers), then minimizing the cross-entropy or
other classification losses forces the parameters of the model to diverge in
norm to infinity \cite{ll19}. In this setting, convergence of the secants at infinity
is important. In \cite{jt20} the authors manage to establish
that for a certain type of prediction functions ($L$-homogeneous and definable in the log-exp
structure) \eqref{eq:thom-inf} holds. For the time being, no further results have been reported.
\smallskip

In a nutshell, proving the gradient conjecture (respectively, the gradient conjecture at infinity) seems to require at least the K{\L}-inequality \eqref{eq:2} together with other properties of o-minimal functions, but it is still unknown if these conjectures are true for general o-minimal functions. \smallskip

In this work we present an example of a smooth convex function in $\mathbb{R}^{2}$, which is real-analytic outside zero (its unique critical point), it satisfies the {\L}ojasiewicz inequality \eqref{eq:1} and fails the gradient conjecture both at zero and at infinity. In
particular, all gradient orbits spiral both at zero and at infinity, underlying
in this way the two failures of o-minimality of the function, despite the fact that the function is convex and satisfies the {\L}ojasiewicz gradient inequality. 

\begin{theorem}[main result]
\label{thm-main} For every $k\in\mathbb{N}$, there exists a $\mathcal{C}^k$-convex function
$f:\mathbb{R}^{2}\rightarrow\mathbb{R}$ with a unique minimum at $\mathcal{O}:=(0,0)$ such that: 
\begin{itemize} 
\item[-] $f$ is real analytic on $\mathbb{R}^{2}\setminus\{\mathcal{O}\}$ ;
\item[-] $f$ satisfies the {\L}ojasiewicz inequality at $\mathcal{O}$ and
\item[-] every maximal gradient orbit $\gamma:(-\infty,T)\rightarrow\mathbb{R}^{2}$ of $f$ spirals infinitely many times
 both when $t\rightarrow-\infty$ (around the origin $\mathcal{O}$) and $t\rightarrow T$ (at infinity). As we show in Lemma~\ref{grad-traj}, $T<+\infty$, {\it i.e.}, maximal orbits blow up in finite positive time.
\end{itemize}
\end{theorem}

\noindent Throughout the manuscript, by gradient orbits (or gradient trajectories) we refer to maximal solutions of the ordinary differential equation:
$$\gamma^{\prime}(t)=\nabla f(\gamma(t)).$$ 
In our example, the function $f$ will be convex, with unique critical point (global minimizer) at~$\mathcal{O}$, where we tacitly assume that $\gamma(0)\not =\mathcal{O}$ (avoiding stationary orbits). \smallskip

Let us briefly describe our strategy for the construction of this example: in Section~\ref{sec:foliation} we prescribe a family of convex sets, all being delimited by ellipses, centered at the origin, and obtained via rotations and size adjustments of a basic ellipse $E(0)$. This is done in a way that convex foliation is obtained, which can be represented by some (quasiconvex) function.  \smallskip

In Section~\ref{sec:cvx-fct}, we further calibrate the parameters so that we can apply a criterium due to de Finetti~\cite{definetti49} and Crouzeix~\cite{crouzeix80} that guarantees that the aforementioned quasiconvex function is in fact convex. The construction yields that the function is real-analytic on $\mathbb{R}^2\setminus\mathcal{O}$, which of course cannot be further improved to real analycity on the whole space, due to the proof of Thom's gradient conjecture~\cite{kmp00}. Instead, we are able to show that the function can be taken $\mathcal{C}^{k}$-smooth at $\mathcal{O}$ for arbitrary large $k\in\mathbb{N}$. Still our construction fails to ensure $\mathcal{C}^{\infty}$. Finally, applying a result of \cite{bdlm10} which gives conditions for a convex function to satisfy~\eqref{eq:2}, we show that our function satisfies K{\L}-inequality and in fact even \eqref{eq:1} (the {\L}ojasiewicz inequality). \smallskip

Gradient orbits are perpendicular to the foliation and explicit calculations, conducted in  Section~\ref{sec:traj}, show that the orbits turn around both at the origin and at infinity, which disproves the conjecture. An additional difficulty to establish spirality is that the evolution of the spherical part of the orbit (the rotation angle $\alpha(t)$ of $\gamma(t)$ in polar coordinates) is not monotone in time, so that the decrease rate is established in average, see Figure~\ref{dess-traj-osc} and Figure~\ref{dess:fonction-alpha-prime}. For a study of monotonic spiraling of orbits of general analytic vector fields in dimensions 2 and 3, we refer to~\cite{sanz02}. 


\section{Construction of a convex real analytic foliation in $\mathbb{R}^{2}\setminus\{\mathcal{O}\}$.}
\label{sec:foliation}
Let us first consider two smooth increasing functions $a,b:\mathbb{R}\to(0,+\infty)$ for which we assume:
\begin{align}
\label{hypab}
\left\{
\begin{array}
[l]{l} \text{$\underset{t\to +\infty}\lim a(t) \, =\,  \underset{t\to +\infty}\lim b(t) = +\infty$}\medskip \\
\text{$\underset{t\to -\infty}\lim a(t) \, = \, \underset{t\to -\infty}\lim b(t) = 0$ \, and} \medskip \\
\text{\phantom{olivier}$a(t)\, \geq\, b(t)$, \quad for all $t\in\mathbb{R}$.}
\end{array}
\right.
\end{align}
The exact definition of the functions $a(t)$ and $b(t)$ will be given in Lemma~\ref{lem-fct-convexe} (Section~\ref{sec:cvx-fct}).
We also consider the rotation matrix by an angle $t$ denoted by:
\begin{align}
\label{rotatM}R(t)=\left(
\begin{array}
[c]{cc}
\cos t & -\sin t\\
\sin t & \cos t
\end{array}
\right)
\end{align}
For $t\in\mathbb{R}$ and $\theta \in\mathbb{T}:=\mathbb{R}/ 2\pi\mathbb{Z}$ we set
\begin{align*}
m(t,\theta):=(x(t,\theta), \,y(t,\theta))=(a(t) \cos\theta,\, b(t)\sin\theta),
\end{align*}
and
\begin{align}
\label{Mmap}M(t,\theta):=R(t)\,m(t,\theta)\,=\,(X(t,\theta), Y(t,\theta)).
\end{align}
Therefore
\begin{align}
\label{formMXY}\left\{
\begin{array}
[c]{l}
X(t,\theta)= x(t,\theta)\cos t - y(t,\theta)\sin t\, =\, a(t) \cos t \cos\theta- b(t) \sin t \sin\theta \smallskip\\
Y(t,\theta)=\, x(t,\theta)\sin t + y(t,\theta)\cos t \,=\, a(t) \sin t \cos\theta+ b(t) \cos t \sin\theta\,.
\end{array}
 \right.
\end{align}
The subset
\begin{align}
\label{ellEt}\mathcal{E}(t):=\{M(t,\theta) : \theta\in\mathbb{T}\}
\end{align}
is an ellipse with major axis of length $a(t)$ and minor axis of length $b(t)$ (see Figure~\ref{dess_ellipse-rot} for illustration). Notice that $\mathcal{E}(t)$ is the rotation by
angle $t$ of the ellipse 
\begin{align*}
E(t):=\,\big\{m(t,\theta) :\,\theta\in\mathbb{T}\big\}\,\,=\,\,\left\{ \, (x,y)\in\mathbb{R}^{2}:\,\,\,
\frac{x^{2}}{a^{2}(t)}+\frac{y^{2}}{b^{2}(t)}=1\,\right\}  .
\end{align*}

\begin{figure}[ht]
\begin{center}
\includegraphics[width=14cm]{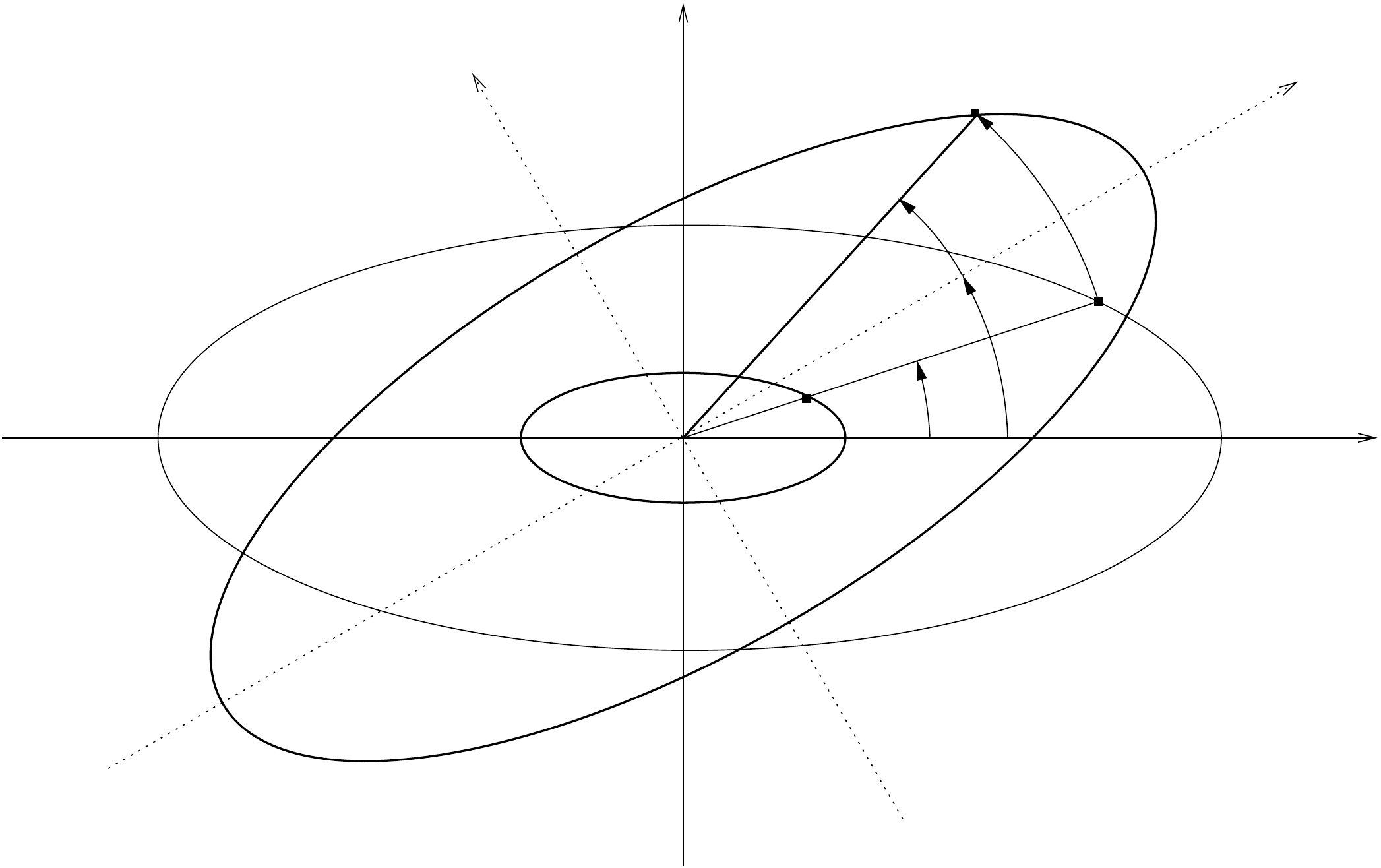}
\caption{The ellipse $\mathcal{E}(t)$ and the map $(t,\theta)\mapsto M(t,\theta)$}
\label{dess_ellipse-rot}
\unitlength=1pt
\begin{picture}(0,0)
\put(-12,163){$\scriptstyle \mathcal{O}$}
\put(41,170){$\scriptstyle m(0,\theta)$}
\put(120,204){$m(t,\theta)$}
\put(75,260){$M(t,\theta)$}
\put(46,155){$\scriptstyle a(0)$}
\put(0,183){$\scriptstyle b(0)$}
\put(154,155){$\scriptstyle a(t)$}
\put(0,225){$\scriptstyle b(t)$}
\put(70,170){$\phi$}
\put(74,222){$\phi$}
\put(94,177){$t$}
\put(104,232){$t$}
\put(191,165){$x$}
\put(2,283){$y$}
\put(90,101){$E(t)$}
\put(-35,70){$\mathcal{E}(t)$}
\end{picture}
\end{center}
\end{figure}
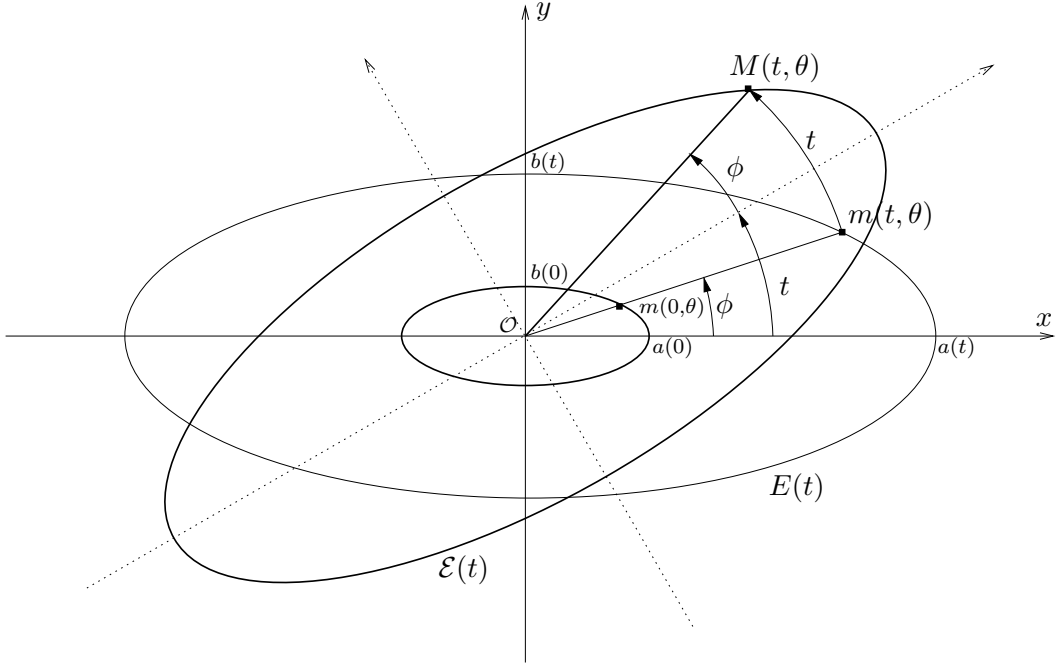

Under an additional condition on the functions $a,b$, the family of ellipses $\{\mathcal{E}(t)\}_{t\in\mathbb{R}}$ defined in~\eqref{ellEt} is disjoint with union equal to $\mathbb{R}^2\setminus\{\mathcal{O}\}$. More precisely, denoting by $a^{\prime}$, $b^{\prime}$ the derivatives of the functions $a$, $b$ respectively, we have the following result:

\begin{lemma} [Convex foliation by ellipses]\label{lem-feuilletage}
Let $a,b:\mathbb{R}\to(0,+\infty)$ satisfy~\eqref{hypab} and assume
\begin{align}
\label{cond-foliation}4\, a(t) \,b(t)\, a^{\prime}(t)\,b^{\prime}(t)\, >\, (a(t)^2-b(t)^2)^{2}\,,\qquad\text{for all }\, t\in\mathbb{R}.
\end{align}
Then $(\mathcal{E}(t))_{t\in\mathbb{R}}$  defines an analytic convex foliation of $\mathbb{R}^{2}\setminus\{\mathcal{O}\}$.
\end{lemma}


\begin{proof}
The proof is divided in three steps: \bigskip \newline
{\it Step 1.} The map $M: \mathbb{R}\times \mathbb{T}\to \mathbb{R}^2\setminus \{\mathcal{O}\}$ is a local analytic diffeomorphism. \medskip\newline
\noindent Indeed, let us first notice that the map $M$, defined by~\eqref{Mmap}--\eqref{formMXY}, is real-analytic as composition of analytic functions. Therefore, if we show that the Jacobian matrix
$\mathcal{J}M
=\left(\begin{array}{cc}\frac{\partial X}{\partial t} & \frac{\partial X}{\partial \theta}\\[1mm]\frac{\partial Y}{\partial t} & \frac{\partial Y}{\partial \theta}\end{array}\right)$
is invertible at each point $(t,\theta)\in \mathbb{R}\times \mathbb{T}$, the assertion follows from the local analytic inverse function theorem~\cite[Theorem~2.5.1]{kp02}.  
To this end, we shall prove that
\begin{eqnarray}\label{eq:2.6}
{\rm det}(\mathcal{J} M)= \frac{\partial X}{\partial t}\,\frac{\partial Y}{\partial \theta} - \frac{\partial Y}{\partial t}\, \frac{\partial X}{\partial \theta}
=\big\langle \frac{\partial M}{\partial t}, n\big\rangle >0,
\end{eqnarray}
where
$n(t,\theta)=-R(\frac{\pi}{2})\,\frac{\partial M}{\partial \theta}= (\frac{\partial Y}{\partial \theta}, -\frac{\partial X}{\partial \theta})$ is the outer unit normal
to the convex set ${\rm conv}\,\mathcal{E}(t)$ (convex envelope of $\mathcal{E}(t)$) at $M(t,\theta)$. Recalling that $M(t,\theta)=R(t)\,m(t,\theta )$ (see~\eqref{Mmap}) and that
the rotation matrix~\eqref{rotatM}
satisfies $$R'(t)=R(t+\frac{\pi}{2}), \quad R(t)^{-1}=R(t)^T=R(-t) \quad \text{and} \quad R(t)\,R(s)=R(t+s),$$
we deduce
\begin{eqnarray*}
\big\langle \frac{\partial M}{\partial t}, n\big\rangle
&=&
\big\langle \frac{\partial}{\partial t}(R(t)m), -R(\frac{\pi}{2})\,\frac{\partial}{\partial \theta} (R(t)m)\,\big\rangle\\
&=&
\big\langle R'(t)m +R(t) \frac{\partial m}{\partial t}, \,-R(\frac{\pi}{2})\,R(t)\, \frac{\partial m}{\partial \theta}\big\rangle\\
&=&
\big\langle R(t+\frac{\pi}{2})\,m +R(t)\frac{\partial m}{\partial t},\, R(t-\frac{\pi}{2})\, \frac{\partial m}{\partial \theta}\big\rangle\\
&=&
\big\langle R(t-\frac{\pi}{2})^T R(t+\frac{\pi}{2})m, \,\frac{\partial m}{\partial \theta}\big\rangle\,
+\, \big\langle  R(t-\frac{\pi}{2})^T R(t)\frac{\partial m}{\partial t}, \, \frac{\partial m}{\partial \theta}\big\rangle\\
&=&
- \big\langle m, \frac{\partial m}{\partial \theta}\big\rangle\, +\,\big\langle R(\frac{\pi}{2})\frac{\partial m}{\partial t},\, \frac{\partial m}{\partial \theta}\big\rangle.
\end{eqnarray*}
Plugging $$\frac{\partial m}{\partial \theta}=(-a\sin\theta , b\cos\theta)\qquad \text{ and } \qquad \frac{\partial m}{\partial t}= (a'\cos\theta, b'\sin\theta)$$
into the above equality, we end up with the expression:
\begin{eqnarray}\label{cond-normale}
{\rm det}(\mathcal{J} M)=\langle \frac{\partial M}{\partial t}, n\rangle = a' b \cos^2\theta + a b' \sin^2\theta + (a^2-b^2)\cos\theta \sin\theta.
\end{eqnarray}
This is a quadratic expression with respect to $\cos\theta$ and $\sin\theta$,
which is positive for all $\theta\in \mathbb{T}$ if and only if the discriminant
$(a^2-b^2)^2-4aa'bb'$ is negative. The result follows in view of~\eqref{cond-foliation}. \bigskip

\noindent{\it Step 2.} The map $M: \mathbb{R}\times \mathbb{T}\to \mathbb{R}^2\setminus\{\mathcal{O}\}$ is injective. \medskip\\
\noindent Fix $t\in \mathbb{R}$.
From~\eqref{eq:2.6}--\eqref{cond-normale}, using compactness of $\mathcal{E}(t)$ and smoothness
of $M$, we deduce the existence of $\delta_t, \rho_t >0$ such that, for all
$s\in [t,t+\delta_t]$, $\theta\in\mathbb{T}$,
\begin{eqnarray*}
\big\langle \,\frac{\partial M}{\partial t}(s,\theta), n(t,\theta)\, \big\rangle \geq \rho_t >0,
\end{eqnarray*}
which yields
\begin{eqnarray*}
\big\langle M(s,\theta)-M(t,\theta), n(t,\theta)\big\rangle \,\geq \,\rho_t(s-t)\, >\,0,
\quad \text{for \,$t<s\leq t+\delta_t$\,\,\text{and }\,$\theta\in\mathbb{T}$.}
\end{eqnarray*}
It follows that ${\rm conv}\,\mathcal{E}(t)\subset {\rm int}\,{\rm conv}\,\mathcal{E}(s)$
for all $s>t$. Therefore, the family $({\rm conv }\,\mathcal{E}(t))_{t\in\mathbb{R}}$ is nested
and the map $M$ is injective. \medskip\\

\noindent {\it Step 3.} The map $M: \mathbb{R}\times \mathbb{T}\to \mathbb{R}^2\setminus\{\mathcal{O}\}$ is surjective. \smallskip\\
\noindent Fix $(x,y)\in \mathbb{R}^2\setminus\{\mathcal{O}\}$ and set, for $t\in\mathbb{R}$ and
$D(t)=\left(\begin{array}{cc} a(t) & 0\\ 0 & b(t)\end{array}\right)$,
\begin{eqnarray*}
\rho(t) &:= &||D(t)^{-1}R(t)^{-1}(x,y)||^2
=\frac{1}{a^2(t)}(x\cos t +y\sin t)^2 +  \frac{1}{b^2(t)}(-x\sin t +y\cos t)^2.
\end{eqnarray*}
We claim that $\rho$ is a smooth decreasing function with
$\displaystyle\lim_{-\infty}\rho= +\infty$ and
$\displaystyle\lim_{+\infty}\rho= 0$. \smallskip

\noindent Indeed, since $(x,y)\not= (0,0)$, we get
$R(t)^{-1}(x,y)\not= (0,0)$ and either
$x\cos t +y\sin t\not= 0$ or $-x\sin t +y\cos t\not= 0$.
Recalling that $a(t),b(t)\to 0$ as $t\to -\infty$, we deduce
$\displaystyle\lim_{-\infty}\rho= +\infty$. We also observe that $\displaystyle\lim_{+\infty}\rho= 0$
is a direct consequence of the fact $a(t),b(t)\to +\infty$ as $t\to +\infty$.\smallskip

\noindent It remains to prove that $\rho'$ is negative. To this end, set $q(t):=x\cos t +y\sin t$ and notice that 
$\rho = a^{-2} q^2 + b^{-2} (q')^2$. Using that $q''=-q$,
we infer
\begin{eqnarray*}
\rho'(t) &= &
-2a'a^{-3}q^2 + 2 a^{-2} q' q - 2b'b^{-3}(q')^2+ 2 b^{-2} q'' q'  \\
&=&
-2 a^{-2} b^{-2} \left( a'a^{-1}b^2 q^2 + (a^2-b^2)qq'+ b'b^{-1}a^2 (q')^2\right).
\end{eqnarray*}
The quadratic expression $a'a^{-1}b^2 q^2 + (a^2-b^2)qq'+ b'b^{-1}a^2 (q')^2$
with respect to $q$ and $q'$
is positive if and only if its discriminant is negative, which
is equivalent, once again, to assume~\eqref{cond-foliation}.
Thus $\rho$ is strictly decreasing and the claim follows.\smallskip

\noindent Using the claim, we infer that there exists a unique $\overline{t}\in\mathbb{R}$
such that $$\rho(\overline{t})=||D(\overline{t})^{-1}R(\overline{t})^{-1}(x,y)||^2=1.$$
Therefore, there exists a unique $\overline{\theta}\in\mathbb{T}$ such that
$D(\overline{t})^{-1}R(\overline{t})^{-1}(x,y)=(\cos\overline{\theta},\sin\overline{\theta})$.
It follows that $M(\overline{t}, \overline{\theta})=(x,y)$, which proves that $M$ is onto.
\end{proof}
\smallskip

\begin{figure}[!h]
\begin{center}
\includegraphics[width=12cm]{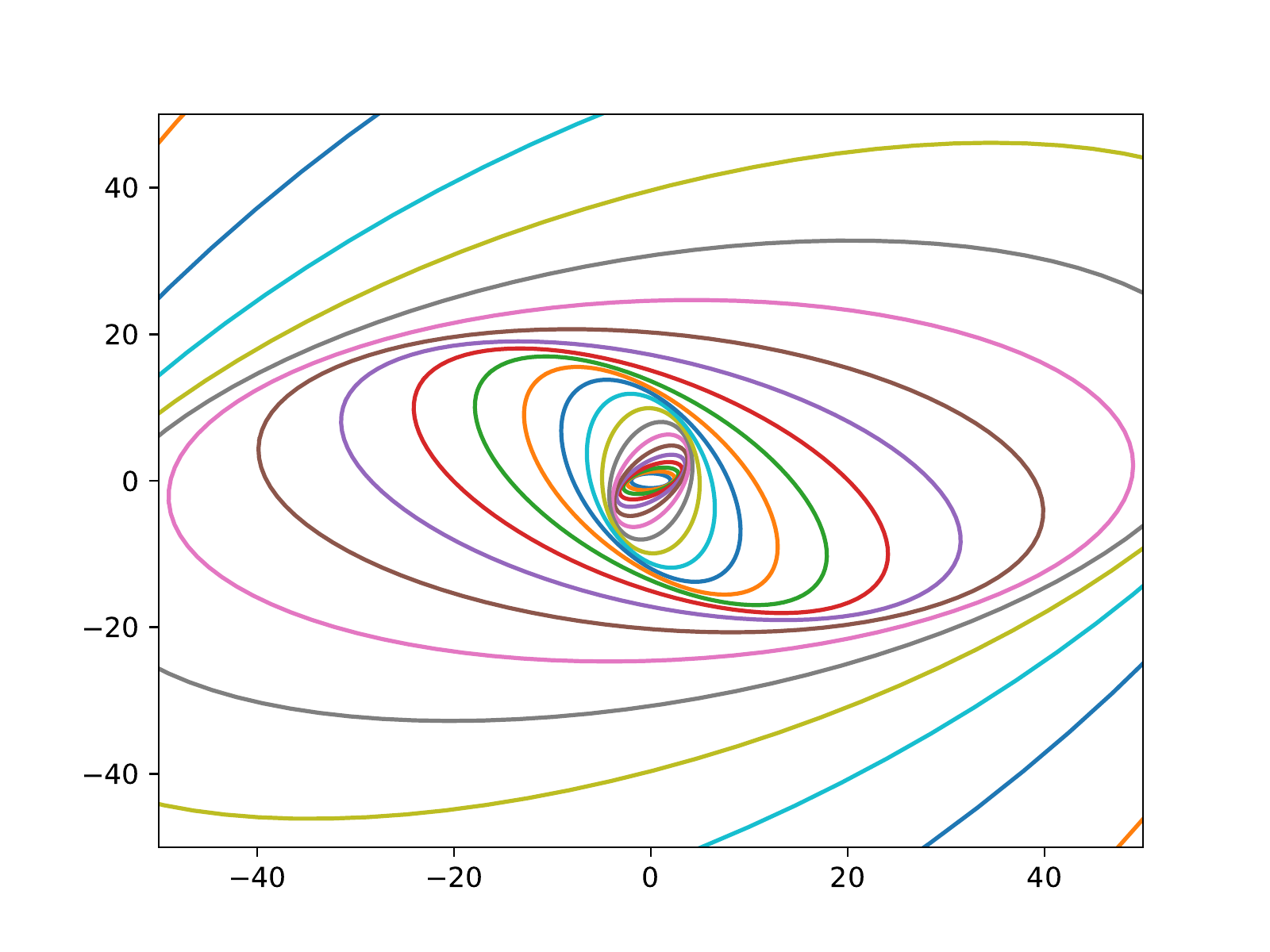}
\end{center}
\caption{The convex foliation $(\mathcal{E}(t))_{t\in\mathbb{R}}$ for $a(t)=2b(t)=2e^t$.}
\label{dess:feuilletage-ellipses}
\end{figure}

A typical instance where Lemma~\ref{lem-feuilletage} applies is to take $a=\mu
b$ for some constant $\mu>1$. Then for
$b(t)=e^{\nu t}$ with $\nu>\frac{\mu^{2}-1}{2\mu}$, it is straightforward to check that
$a,b$ satisfy~\eqref{hypab} and~\eqref{cond-foliation}. Figure~\ref{dess:feuilletage-ellipses} represents the explicit choice
$\mu=2$ and $\nu=1$ leading to $a(t)=2 e^{t}$ and $b(t)=e^{t}$.


\section{Defining the convex function and regularity properties}
\label{sec:cvx-fct}

In this section we shall show that for a more precise choice of the functions $a(t), b(t)$ we can construct a convex function whose level sets are exactly the foliation $\{ \mathcal{E}(t) \}_{t\in\mathbb{R}}$. Moreover, we shall show that this convex function is smooth, real-analytic on 
$\mathbb{R}^2\setminus\{\mathcal{O}\}$ and satisfies~\eqref{eq:1}. \smallskip \\

Concretely, let us denote by $\varphi:\mathbb{R}\to\mathbb{R}$  a smooth strictly increasing function
satisfying $\varphi(-\infty):=\underset{t\rightarrow-\infty}{\lim}\varphi(t)=0$
(the concrete definition of the function $\varphi$ will be given in~\eqref{cond-a-b-varphi-nu}, see Lemma~\ref{lem-fct-convexe}) and let us set for all $M\in\mathbb{R}^{2}$
\begin{align}
\label{def-f}
f(M)=\left\{
\begin{array}
[c]{ll}
\phantom{0}0\,, & \text{if $M=(0,0)$,}\medskip \\
\varphi(t), & \text{if \,$M\in\,\mathcal{E}(t)$},\\
\end{array}
\right.
\end{align}
where $\mathcal{E}(t)$ is the ellipse given in~\eqref{ellEt}. We shall now show that we can adjust the parameters and choose $\varphi$ in a way that~\eqref{def-f} gives a well-defined convex function.


\begin{lemma} [Construction of the convex function]\label{lem-fct-convexe} Setting for $t\in\mathbb{R}$
\begin{align}
\label{cond-a-b-varphi-nu}
\begin{array}
[c]{cc}
a(t)= \sqrt{2} \,\, {\rm exp}({t}), \quad b(t)={\rm exp}(t) &
\text{in~\eqref{formMXY},} \medskip \\
\varphi(t)= {\rm exp}(t/\tau), \quad\, \tau\in(0, \frac{1}{10}), & \text{in~\eqref{def-f}},
\end{array}
\end{align}
the function $f$ defined by~\eqref{def-f} is convex, with level
sets the ellipses $\mathcal{E}(t)$ and $\mathrm{argmin}
\,f=\{\mathcal{O}\}$.
\end{lemma}


\begin{proof}
Since the functions $a,b$ satisfy~\eqref{hypab} and~\eqref{cond-foliation}, we deduce by  Lemma~\ref{lem-feuilletage} that 
${\rm conv }(\mathcal{E}(t))_{t\in\mathbb{R}}$ is a convex foliation. In particular, the function $f$ is well defined from~\eqref{def-f} with sublevel sets
$$
[f\leq \lambda]:=\{M\in\mathbb{R}^2:\,f(M)\leq \lambda\}=\mathrm{conv }\, [\mathcal{E}(\varphi^{-1}(\lambda))]
= \mathrm{conv }\, [\mathcal{E}(\tau\log \lambda)]
$$
compact and convex. Therefore $f$ is a coercive, quasiconvex function. 
\smallskip\\
\noindent We shall now use a result due to de Finetti and Crouzeix~\cite{definetti49,crouzeix80} which asserts that the quasiconvex function $f$ is convex if and only if $$\lambda\mapsto \sigma_{[f\leq \lambda]}(p) \,\,\text{is concave for every }\,p\in\mathbb{R}^2,$$
where $\displaystyle\sigma_A (p)=\max_{M\in A}\, \langle p, M\rangle$ is the support function to the subset $A$.
Without loss of generality, we may restrict to unit vectorss $p\in\mathbb{R}^2$, which results in assuming that 
$p=(\cos\alpha , \sin\alpha )$, for some $\alpha\in\mathbb{T}$. Therefore, we are led to prove that the function
\begin{eqnarray*}
G_\alpha (\lambda) &:=& \sup \left\{ \Big\langle (x,y),(\cos\alpha , \sin\alpha )\Big\rangle : \,f(x,y)\leq \lambda \right \}\\
&=&
\sup \left\{ \Big\langle M(t,\theta),(\cos\alpha , \sin\alpha )\Big\rangle : \,f(M(t,\theta))=\varphi(t)\leq \lambda\right \}\\
&=&
\max \left \{ \Big\langle M(t,\theta),(\cos\alpha , \sin\alpha )\Big\rangle : \theta\in\mathbb{T}, t=t(\lambda)=\varphi^{-1}(\lambda) \right \}
\end{eqnarray*}
is concave. To this, end, after straightforward calculations we obtain
\begin{eqnarray*}
\Big\langle M(t,\theta),(\cos\alpha , \sin\alpha )\Big\rangle
&=&\Big\langle R(t)\,m(t,\theta),\,(\cos\alpha , \sin\alpha )\Big\rangle\\
&=&\Big\langle (a(t)\cos\theta,\, b(t)\sin\theta),\, R(-t)\,(\cos\alpha , \sin\alpha )\Big\rangle\\
&=& \Big\langle (\cos\theta, \sin\theta ),\,\left(a(t)\cos(\alpha -t) , b(t)\sin (\alpha-t)\right)\Big\rangle
\end{eqnarray*}
whence we deduce
\begin{eqnarray}\label{eq:G_a}
G_\alpha (\lambda) \,=\, \Big\Vert \,a(t(\lambda))\cos(\alpha -t(\lambda)) ,\, b(t(\lambda))\sin (\alpha-t(\lambda))\,\Big\Vert
\,=\, \sqrt{g_\alpha (\lambda)}
\end{eqnarray}
with
\begin{eqnarray}\label{eq:g_a}
g_\alpha (\lambda)  =a(t(\lambda))^2\cos^2(t(\lambda)-\alpha)\,+\, b(t(\lambda))^2\sin^2(t(\lambda)-\alpha).
\end{eqnarray}
Calculating the second derivative of $G_{\alpha}$ in~\eqref{eq:G_a} yields
\begin{eqnarray*}
G_\alpha''=\frac{2 g_\alpha'' g_\alpha -(g_\alpha')^2}{4 g_\alpha^{3/2}}.
\end{eqnarray*}
Therefore, the functions $\{G_\alpha\}_{\alpha\in \mathbb{T}}$ are concave provided we establish: 
\begin{equation}\label{eq:con}
2 g_\alpha'' g_\alpha -(g_\alpha')^2\leq 0, \quad\text{for all } \alpha\in\mathbb{T}\,.
\end{equation}
At this step, we replace in~\eqref{eq:g_a} the choice for $a$, $b$ and $\varphi$ given in~\eqref{cond-a-b-varphi-nu}:
$$ a(t)=\sqrt{2} \,e^t\,,  \quad b(t)=e^t \quad\text{and }\, \lambda= \varphi(t)=e^{t/\tau},\quad\text{for all }\, t\in\mathbb{R},$$
and we seek for the values of $\tau>0 $ that ensure inequality~\eqref{eq:con}.  
In particular, 
$$ t:=t(\lambda)=\tau\log\lambda, \quad \text{whence }\, t'(\lambda)=\frac{\tau}{\lambda} \,\,\text{and }\, t''(\lambda)=-\frac{\tau}{\lambda^2}<0.$$
After tedious computations, we get
$$  g_\alpha= e^{2t}\left(\cos^2(t\!-\!\alpha)+1\right),\qquad 
g_\alpha' = 2 \,e^{2t}\, t' \left (\cos^2(t\!-\!\alpha)+1-\cos(t\!-\!\alpha)\sin(t\!-\!\alpha)\right) $$
and 
\[  
g_\alpha'' = 2 e^{2t}\,\Big( (t')^2 \big( 3-4\cos(t\!-\!\alpha)\sin(t\!-\!\alpha)\big) + t'' \left( \cos^2(t\!-\!\alpha)+1 -\cos(t\!-\!\alpha)\sin(t\!-\!\alpha)\right)\Big ).
\]
Hence \medskip \\ $2 g_\alpha'' g_\alpha - (g_\alpha')^2 =$
\begin{eqnarray*}
&=& 4 e^{4t}(t')^2\Big\{ \big( \cos^2(t\!-\!\alpha)+1\big)\big(3- 4 \cos(t\!-\!\alpha)\sin(t\!-\!\alpha)\big ) -\big(\cos^2(t\!-\!\alpha)+1- \cos(t\!-\!\alpha)\sin(t\!-\!\alpha)\big)^2\medskip \\
&\phantom{ley}&\hspace{4cm} +\,\,4\, e^{4t}\,t''\, \big( \cos^2(t\!-\!\alpha)+1 \big)\big(\cos^2(t\!-\!\alpha)+1 - \cos(t\!-\!\alpha)\sin(t\!-\!\alpha) \big)\Big\} \smallskip\\
&\leq&
4 e^{4t} \Big( 5 (t')^2 +\frac{1}{2}t'' \Big)\quad\leq\quad
\frac{2\tau (10\tau -1)e^{4t}}{\lambda^2},
\end{eqnarray*}
which is negative provided we choose $\tau < 1/10$.
\end{proof}


\bigskip

\noindent We fix $M:\mathbb{R}\times\mathbb{T}\mapsto
\mathbb{R}^{2}\backslash\{\mathcal{O}\}$ under the choice made in
Lemma~\ref{lem-fct-convexe}, that is,
\begin{equation}
M(t,\theta)=(X(t,\theta),Y(t,\theta))=e^{t}\left(  \sqrt{2}\cos t\cos
\theta-\sin t\sin\theta,\sqrt{2}\sin t\cos\theta+\cos t\sin\theta
\,\right)  . \label{def-M}
\end{equation}
Setting
\begin{equation}
\left\{
\begin{array}
[c]{l}
\tilde{f}:\mathbb{R}\times\mathbb{T}\mapsto\mathbb{R}\medskip \\
\tilde{f}(t,\theta)=\varphi(t)=\exp(t/\tau)
\end{array}
\right.
\label{ftilde}
\end{equation}
we observe that the convex function $f$ defined in~\eqref{def-f}
satisfies:
\begin{equation}
f(x,y)=\left\{
\begin{array}
[c]{cc}
(\tilde{f}\circ M^{-1})(x,y), & \text{if \thinspace$(x,y)\neq \mathcal{O}$},\smallskip \\
\phantom{0}0\,, & \text{if $(x,y)=\mathcal{O}$.}
\end{array}
\right.  \label{def-f-comp}
\end{equation}
With the next couple of lemmas we show that the function $f$, apart from being convex, enjoys several other good properties.

\begin{lemma} [Properties of the convex function]\label{lem-regularite-fct} Let
$f:\mathbb{R}^{2}\mapsto\lbrack0,+\infty)$ be the convex function
defined by~\eqref{def-M}--\eqref{def-f-comp} for $0<\tau<1/10$. Then
\begin{itemize}
\item[{\rm (i).}] $f$ is strictly positive on $\mathbb{R}^{2}\setminus\{\mathcal{O}\}$ with $f(\mathcal{O})=0\,.$

\item[{\rm (ii).}] For all $(x,y)\in\mathbb{R}^{2}$, it holds
\begin{align}
\label{ineg-f}
\left( 1/\sqrt{2}\right)^{1/\tau}\big \Vert(x,y)\big\Vert^{1/\tau} \,\leq \, f(x,y) \leq
\,\big\Vert (x,y)\big\Vert ^{1/\tau}.
\end{align}
In particular, $f$ is coercive.

\item[{\rm (iii).}] $f$ is real analytic on $\mathbb{R}^{2}\setminus\{\mathcal{O}\}$ and $f\in \mathcal{C}^{1}(\mathbb{R}^{2})$\,.

\item[{\rm (iv).}] $f$ satisfies the {\L}ojasiewicz inequality~\eqref{eq:1} with $\vartheta=1-\tau$, $c=\tau/\sqrt{2}$, $a\equiv\mathcal{O}$ and $f(\mathcal{O})=0$, that is
\begin{align}
\label{ineg-loja1}\Vert \nabla f (x,y)\Vert \, \geq\, \left(\frac{\tau}{\sqrt{2}}\right)\, f(x,y)^{1-\tau}\,,
\qquad\text{for all $(x,y) \in\mathbb{R}^{2}$.}
\end{align}

\end{itemize}
\end{lemma}

\begin{proof}
(i). It is straightforward from the definition of $f$ in~\eqref{def-f} and the choice of $\varphi$.\smallskip\\

\noindent (ii). From Lemma~\ref{lem-feuilletage}, for every $(x,y)\in \mathbb{R}^2\setminus\{\mathcal{O}\}$, there exists a unique $t\in\mathbb{R}$
such that $(x,y)\in\mathcal{E}(t)$ and we have
\begin{eqnarray*}
\frac{x^2+y^2}{a^2(t)} \,
\leq\,\,
\frac{1}{a^2(t)}(x\cos t +y\sin t)^2 +  \frac{1}{b^2(t)}(-x\sin t +y\cos t)^2
\,=\, 1\,\,
\leq \,\,\frac{x^2+y^2}{b^2(t)},
\end{eqnarray*}
whence
\begin{eqnarray*}
e^t=b(t)\leq\, \Vert(x,y)\Vert\,\leq a(t)=\sqrt{2}e^t.
\end{eqnarray*}
We deduce easily that
\begin{eqnarray*}
2^{-1/(2\tau)}\,\Vert(x,y)\Vert^{1/\tau} \,\leq \, f(x,y)=\varphi(t)=e^{t/\tau}\, \leq\, \Vert(x,y)\Vert^{1/\tau}.
\end{eqnarray*}

\noindent (iii).  It follows from~\eqref{def-f} that $f=\varphi\,\circ\, p_1\,\circ M^{-1}$ on $\mathbb{R}^2\setminus\{\mathcal{O}\}$, where $p_1: \mathbb{R}\times \mathbb{T} \mapsto \mathbb{R}$ with $p_1(t,\theta)= t$. By Lemma~\ref{lem-feuilletage}, the map $M:\mathbb{R}\times \mathbb{T}\mapsto \mathbb{R}^2\setminus\{\mathcal{O}\}$
given in~\eqref{def-M} is a real analytic diffeomorphism. Since $p_1$ and 
$\varphi$ are analytic, the first part of the assertion follows. In particular, the function $f$ is $\mathcal{C}^\infty$-smooth on $\mathbb{R}^2\setminus\{\mathcal{O}\}$. \smallskip\\
Since $1/\tau >1$, the function $(x,y)\mapsto \Vert(x,y)\Vert^{1/\tau}$ is~$\mathcal{C}^1$ over $\mathbb{R}^2$ and~\eqref{ineg-f} yields that $f$ is differentiable at $\mathcal{O}$ with $\nabla f(\mathcal{O})=0$.
Therefore $f$ is differentiable everywhere in $\mathbb{R}^2$ and,
since it is convex, it is $\mathcal{C}^1$ (see for instance, \cite[p.~20]{phelps89}).

\bigskip\noindent(iv) Since $S:=\mathrm{argmin}\,f=\{\mathcal{O}\}$, we have
$\mathrm{dist}_{S}(M)=\Vert M\Vert$ for all $M\!=\!(x,y)\!\in
\!\mathbb{R}^{2}$. Therefore, the first inequality in~\eqref{ineg-f} can be
written
\[
f(M)\,\geq\,\mathbf{m}\!\left(  \mathrm{dist}_{S}(M)\right)  \quad\text{for all
$M\in\mathbb{R}^{2}$},
\]
where $\mathbf{m}(r)=2^{-1/(2\tau)}\,r^{1/\tau}$. Since
\[
\frac{\mathbf{m}^{-1}(s)}{s}=\sqrt{2}\,s^{\tau-1}\in L_{\mathrm{loc}}^{1}%
((0,+\infty)),
\]
we deduce from \cite[Theorem~30]{bdlm10} that the K{\L }-inequality
\[
\Vert\nabla(\psi\circ f)(M)\Vert\geq1,
\]
holds for all $M\in\lbrack f>0]:=\mathbb{R}^{2}\setminus\{\mathcal{O}\}$,
where
\[
\psi(s)=\int_{0}^{s}\frac{\mathbf{m}^{-1}(\sigma)}{\sigma}\,d\sigma\,=\,\frac{\sqrt{2}%
}{\tau}\,s^{\tau}.
\]
A straightforward calculation shows that~\eqref{ineg-loja1} holds.
\end{proof}

\begin{lemma}
[$\mathcal{C}^k$-smoothness of the convex function]\label{lem-Ck-fct} Let
$f$ be the convex function
defined by~\eqref{ftilde}--\eqref{def-f-comp} for $0<\tau<1/10$.
Let $k\in\mathbb{N}$ be the biggest integer such that $k<\frac{1}{\tau}$.
Then $f\in \mathcal{C}^k(\mathbb{R}^2)$ and $f\not\in \mathcal{C}^{k+1}(\mathbb{R}^2)$.
\end{lemma}

\begin{proof}
Recalling that $f$ is real analytic in $\mathbb{R}^2\setminus\{\mathcal{O}\}$ with $f(\mathcal{O})=0$ and $\nabla f(\mathcal{O})=0$, in order
to prove that $f$ is $\mathcal{C}^{k}$, it is sufficient to show
that all the partial derivatives
\begin{equation}
\frac{\partial^{l_1+l_2} f}{\partial x^{l_1}\partial y^{l_2}}, \quad l_1+l_2\leq k,
\end{equation}  
which exist in $\mathbb{R}^{2}\setminus\{\mathcal{O}\}$, converge to $0$ at
$\mathcal{O}$. To this end, it is more convenient to start by computating the partial derivatives of
$\tilde{f}$ defined in~\eqref{ftilde}. We have
\begin{eqnarray*}
\tilde{f}(t,\theta):=f(M(t,\theta))=e^{t/\tau}=f(x,y) \quad\text{for $(x,y)=M(t,\theta)=(X(t,\theta),Y(t,\theta))$,}
\end{eqnarray*}
and by differentiation, we obtain
\begin{eqnarray}\label{diff1-f}
\left(\begin{array}{c}\frac{\partial\tilde{f}}{\partial t}\\[1mm]\frac{\partial\tilde{f}}{\partial \theta}\end{array}\right)
= \left(\begin{array}{c} \frac{1}{\tau}e^{t/\tau} \\0 \end{array}\right)
= 
\left(\begin{array}{cc}\frac{\partial X}{\partial t} &\frac{\partial Y}{\partial t}\\[1mm]
\frac{\partial X}{\partial \theta} & \frac{\partial Y}{\partial \theta}
\end{array}\right)
\left(\begin{array}{c}\frac{\partial f}{\partial x}\\[1mm]\frac{\partial f}{\partial y}\end{array}\right).
\end{eqnarray}
We can compute explicitely the partial derivatives of $X$ and $Y$, see~\eqref{def-M}, to obtain
\begin{eqnarray*}
\frac{\partial X}{\partial t}, \frac{\partial Y}{\partial t}, \frac{\partial X}{\partial \theta}, \frac{\partial Y}{\partial \theta}
= e^t P(t,\theta),  
\end{eqnarray*}
where $P(t,\theta)$ denotes generically a smooth periodic (hence bounded) function with respect to~$t$ and~$\theta$.
More generally, in what follows, $P_{n,m}(t,\theta)$ (respectively $B_{n,m}(t,\theta)$) denotes a $n\times m$ matrix,
the coefficients of which are smooth and periodic with respect to $t$ and $\theta$ (respectively bounded in $(-\infty,1]\times\R$). It follows that
\begin{eqnarray*}
\left(\begin{array}{c}\frac{\partial f}{\partial x}\\[1mm]\frac{\partial f}{\partial y}\end{array}\right)
= \frac{1}{\frac{\partial X}{\partial t}\frac{\partial Y}{\partial \theta}-\frac{\partial Y}{\partial t}\frac{\partial X}{\partial \theta}}
\left(\begin{array}{cc} \frac{\partial Y}{\partial \theta} &-\frac{\partial Y}{\partial t}\\[1mm]
-\frac{\partial X}{\partial \theta} & \frac{\partial X}{\partial t}
\end{array}\right)
\left(\begin{array}{c} \frac{1}{\tau}e^{t/\tau} \\0 \end{array}\right)
\end{eqnarray*}
Since
\begin{eqnarray*}
0<\,e^{2t}(\sqrt{2}-\frac{1}{2})\, \leq \,\,\frac{\partial X}{\partial t}\frac{\partial Y}{\partial \theta}-\frac{\partial Y}{\partial t}\frac{\partial X}{\partial \theta} \, \,=\, e^{2t}(\sqrt{2}+\cos\theta \sin\theta ) \, \leq \,e^{2t}(\sqrt{2}+\frac{1}{2}),
\end{eqnarray*}
we obtain
\begin{eqnarray}\label{form-grad-f}
\left(\begin{array}{c}\frac{\partial f}{\partial x}\\[1mm]\frac{\partial f}{\partial y}\end{array}\right)
=  e^{(\frac{1}{\tau} -1)t} P_{2,1}(t,\theta),
\end{eqnarray}
from which we infer that $\frac{\partial f}{\partial x}, \frac{\partial f}{\partial y}\to 0$ as $(x,y)\to \mathcal{O}$ or equivalently as $t\to -\infty$,
since $\frac{1}{\tau} >1$.
We then recover the fact that $f$ is $\mathcal{C}^1$, with $\nabla f(\mathcal{O})=(0,0)$.\smallskip

To prove that $f$ is $\mathcal{C}^2$ (when $\frac{1}{\tau} >2$), we differentiate again~\eqref{diff1-f} to obtain
\begin{eqnarray}\label{diff-f-2}
\left(\begin{array}{c}\frac{\partial^2\tilde{f}}{\partial t^2}\\[1mm]\frac{\partial^2\tilde{f}}{\partial t\partial\theta}\\[1mm]
\frac{\partial^2\tilde{f}}{\partial \theta^2} \end{array}\right)
= \left(\begin{array}{c} \frac{1}{\tau^2}e^{t/\tau} \\0\\0 \end{array}\right)
=
e^{2t} P_{3,3}(t,\theta)
\left(\begin{array}{c}\frac{\partial^2 f}{\partial x^2}\\[1mm]\frac{\partial^2 f}{\partial x\partial y}\\[1mm]
\frac{\partial^2 f}{\partial y^2}\end{array}  \right)
+ e^t P_{3,2}(t,\theta)
\left(\begin{array}{c}\frac{\partial f}{\partial x}\\[1mm]\frac{\partial f}{\partial y}\end{array}\right),
\end{eqnarray}
where the coefficients of $e^{2t} P_{3,3}(t,\theta)$ are of the form
\begin{eqnarray*}
Z_1 Z_2, \quad \text{with } Z_1,Z_2\in \mathcal{D}_1:=\Big\{ \frac{\partial X}{\partial t}, \frac{\partial Y}{\partial t}, \frac{\partial X}{\partial \theta}, \frac{\partial Y}{\partial \theta}\Big\}
\end{eqnarray*}
and the coefficients of $e^t P_{3,2}(t,\theta)$ are second derivatives of $X$, $Y$.
The matrix $P_{3,3}(t,\theta)$ is invertible since $(t,\theta)\in\R\times \T\mapsto M(t,\theta):=(x,y)\in \R^2\setminus\{\mathcal{O}\}$
is an analytic diffeomorphism.
Finally, we get
\begin{eqnarray*}
\left(\begin{array}{c}\frac{\partial^2 f}{\partial x^2}\\[1mm]\frac{\partial^2 f}{\partial x\partial y}\\[1mm]
\frac{\partial^2 f}{\partial y^2}\end{array}  \right)
= e^{(\frac{1}{\tau}-2)t}  P_{3,1}(t,\theta) + e^{(\frac{1}{\tau}-1)t} B_{3,1}(t,\theta),
\end{eqnarray*}
which proves that the second derivatives of $f$ converge to 0 as $(x,y)\to \mathcal{O}$
if $\frac{1}{\tau} >2$. Therefore $f$ is $\mathcal{C}^2$ with $\nabla^2 f(\mathcal{O})=0_{2\times 2}$. \smallskip

Continuing along the same lines, when differentiating $l$ times, the invertible
matrix in front of the $l$-th order derivatives of $f$ has coefficients of
the form $Z_1Z_2\cdots Z_l$ with $Z_1,\cdots , Z_l\in  \mathcal{D}_1$ and,
after tedious computations,  we obtain
\begin{eqnarray}\label{rel-fin-l}
\left(\begin{array}{c}\frac{\partial^l f}{\partial x^l}\\[1mm]\vdots\\[1mm]
\frac{\partial^l f}{\partial x^{l-i}\partial y^{i}}\\[1mm]\vdots\\[1mm]    
\frac{\partial^l f}{\partial y^l}\end{array}  \right)
= e^{(\frac{1}{\tau}-l)t}  P_{l+1,1}(t,\theta) + e^{(\frac{1}{\tau}-(l-1))t} B_{l+1,1}(t,\theta),
\end{eqnarray}
which converges to $0$ as $(x,y)\to \mathcal{O}$ as long as  $\frac{1}{\tau} >l$.
Therefore $f$ is $\mathcal{C}^l$ and all the $l$-th order derivatives of $f$ are
zero at $\mathcal{O}$ and we conclude that $f\in \mathcal{C}^k(\R^2)$, where $k$ is the biggest integer
such that $k< \frac{1}{\tau}$. \smallskip

Let us now assume, towards a contradiction, that $f$ is $\mathcal{C}^{k+1}$. Then we can write a Taylor expansion of $f$  up to the order $k+1$ at $\mathcal{O}$. Since $\nabla^l f(\mathcal{O})=0$ for $l\leq k$, we obtain that
\begin{eqnarray}\label{allure-f-0}
f(x,y)=O(||(x,y)||^{k+1})\quad \text{in a neighborhood of $\mathcal{O}$},
\end{eqnarray}  
where $O(r^{k+1})/r^{k+1}$ is bounded near $0$. If $\frac{1}{\tau}\not\in\N$, then
$k+1 > \frac{1}{\tau}$, and we obtain a straightforward contradiction with the first inequality in~\eqref{ineg-f}.
If now $k+1 = \frac{1}{\tau}\in\N$, then~\eqref{allure-f-0} is not anymore contradictory
with~\eqref{ineg-f}. But writing~\eqref{rel-fin-l} with $l=k+1$, we get
\begin{eqnarray*}
\left(\begin{array}{c}\frac{\partial^{k+1} f}{\partial x^{k+1}}\\[1mm]\vdots\\[1mm]
\frac{\partial^{k+1} f}{\partial y^{k+1}}\end{array}  \right)
=  P_{k+2,1}(t,\theta) + e^{t} B_{k+2,1}(t,\theta).
\end{eqnarray*}
The second term above converges to zero as $t\to -\infty$, or equivalently as $(x,y)\to\mathcal{O}$,
but $P_{k+2,1}(t,\theta)$ is a periodic nonconstant matrix with respect to $t$ and $\theta$ so
cannot converge as $t\to -\infty$, contradicting our assumption. This ends the proof.
\end{proof}


\section{Oscillating gradient trajectories}

\label{sec:traj}

Let us start by showing that maximal gradient orbits blow up in finite positive time (and converge to the unique minimum $\mathcal{O}$ of the convex function $f$ as $t\to-\infty$). 

\begin{lemma}
[Gradient trajectories of the convex function]\label{grad-traj}
Let $f$ be the convex function defined in Lemma~\ref{lem-fct-convexe}.
Then the ordinary differential equation for the gradient orbits
\begin{align}
\label{edo-gradient}\left\{
\begin{array}
[c]{ll}
\gamma^{\prime}(t) = \nabla f(\gamma(t) ), & t \in\mathbb{R},\medskip \\
\gamma(0)=\gamma_0 \in\mathbb{R}^{2}\setminus\{\mathcal{O}\}. &
\end{array}
\right.
\end{align}
admits a unique maximal solution $\gamma$ defined in $(-\infty, T)$ such that
$$
\underset{t\to-\infty}{\rm lim} \gamma(t)=\mathcal{O}
$$
and $\gamma$ blows up in a finite time $$T\leq \frac{2^{1/2\tau }}{(\frac{1}{\tau}-2)\, \Vert\gamma_0\Vert^{\frac{1}{\tau}-2}}\qquad (0< \tau <\frac{1}{10}\quad\text{is introduced
in~\eqref{cond-a-b-varphi-nu}}),$$ i.e.,
$$
\underset{t\nearrow T}{\rm lim}\,\Vert \gamma (t)\Vert =+\infty.
$$
\end{lemma}

\begin{proof}
Since $f$ is $\mathcal{C}^k$ with $k\geq 2$ (Lemma~\ref{lem-Ck-fct}),
there exists a unique maximal solution of~\eqref{edo-gradient}, denoted by $\gamma\in \mathcal{C}^k((S, T))$, where $-\infty\leq S<0<T\leq +\infty$. The function $f$ being convex and
coercive with a unique minimum at $\mathcal{O}$, we infer that $S=-\infty$
and $\gamma (t)\to \mathcal{O}$ as $t\to -\infty$.
In particular, $\gamma (t)\not= \mathcal{O}$ for every $t\in (-\infty, T)$
and consequently the function $t\mapsto z(t):=\Vert \gamma(t)\Vert$ is differentiable. Using the convexity of $f$ and~\eqref{ineg-f}, 
we deduce:
\begin{eqnarray*}
\frac{d}{dt}\Vert \gamma(t)\Vert =\langle \gamma'(t), \frac{\gamma(t)}{\Vert \gamma(t)\Vert}\rangle
=   \langle \nabla f(\gamma(t)), \frac{\gamma(t)}{\Vert \gamma(t)\Vert}\rangle
\geq \frac{f(\gamma(t))}{\Vert \gamma(t)\Vert}\geq 2^{-\frac{1}{2\tau}} \Vert  \gamma(t)\Vert^{\frac{1}{\tau}-1}.
\end{eqnarray*}
It follows that
\begin{eqnarray*}
\Vert \gamma(t)\Vert \geq \frac{1}{\left( \Vert \gamma_0\Vert^{2- \frac{1}{\tau}} - 2^{-\frac{1}{2\tau}}(\frac{1}{\tau}-2)t\right)^\frac{\tau }{1-2\tau}},
\end{eqnarray*}
where the above right-hand side is the exact solution to the scalar ordinary differential equation
$z'(t)= 2^{-\frac{1}{2\tau}} z(t)^{\frac{1}{\tau}-1}$, $z(0)=\Vert \gamma_0\Vert$.
We conclude that the maximal solution $\gamma$ blows up in finite positive time.
\end{proof}

In fact, finding gradient orbits is a geometric problem. We seek the unique curve $\gamma$ passing through $\gamma_0$,
which is orthogonal to the level sets of $f$. It is convenient to parametrize $\gamma$ as
\begin{align}
\label{param-s}\gamma(s)= M(t(s),\theta(s))= (X(t(s),\theta(s)),Y(t(s),\theta
(s))), \ s\in\mathbb{R}
\end{align}
using the notations~\eqref{Mmap}--\eqref{formMXY}. Under this parametrization $\gamma(s)\!\in\!\mathcal{E}(t(s))$, for every
$s\in\mathbb{R}$ and $\gamma^{\prime}(s)$ is a normal vector at $\gamma(s)$ to the (convex) sublevel set $[f\!\leq\!f(\gamma(s))]=\mathrm{conv }\,\mathcal{E}(t(s))$. Therefore:
\begin{align}
\label{cond-orth}\gamma^{\prime}(s) \perp\partial_{\theta}M(t(s),\theta(s)),
\quad\text{for all $s\in\mathbb{R}$}.
\end{align}
We define the rotation angle $s\mapsto \alpha(s)$ as the angle between the $x$-axis and the secant $\frac{\gamma
(s)}{\Vert \gamma(s)\Vert}$ (spherical part of the orbit) varying in a continuous way. Therefore 
\begin{align*}
\left\{
\begin{array}
[c]{l}
\cos\alpha(s)=\frac{X(t,\theta)}{\sqrt{X(t ,\theta)^2+Y(t ,\theta)^2}}\,,\medskip\\
\sin\alpha(s) =\frac{Y(t ,\theta)}{\sqrt{X(t ,\theta)^2+Y(t ,\theta)^2}}.
\end{array}
\right.
\end{align*}
In particular, according to the notation used in \eqref{Mmap}--\eqref{ellEt}, if $\phi(s)$ is the angle in polar coordinates of the point $m(t,\theta)$, then we have (see Figure~\ref{dess_ellipse-rot}):
$$\alpha(s)=t(s)+\phi(s),\quad\text{for all }\, s\in\mathbb{R}.$$

\begin{figure}[ht]
\begin{center}
\includegraphics[width=5cm]{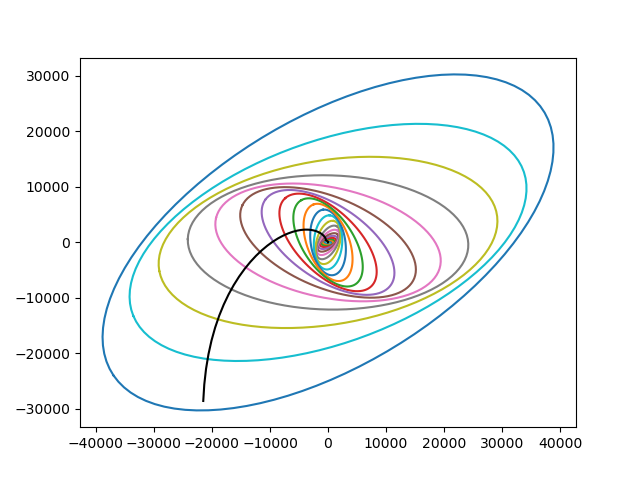}\hspace*{-0.5cm}
\includegraphics[width=5cm]{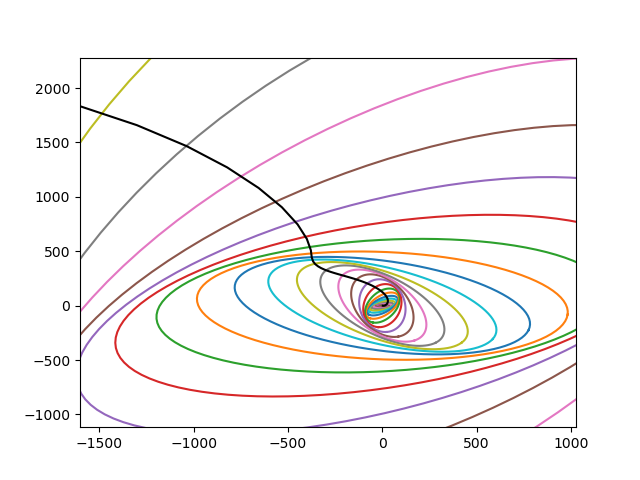}\hspace*{-0.5cm}
\includegraphics[width=5cm]{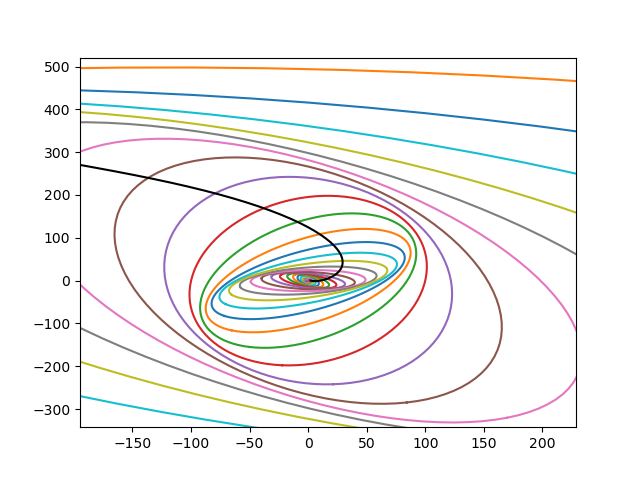}
\end{center}
\caption{\small Gradient orbit $\gamma(s)$ with initial point $\gamma(0)=(2,0)$, then zoom and extra-zoom.}
\label{dess-traj-osc}
\end{figure}

\begin{lemma}
[Spiraling around the origin]\label{lem-traj-oscill} Let $f$ be the convex function defined in~\eqref{def-f} under the assumption~\eqref{cond-a-b-varphi-nu} and let $s\mapsto \gamma(s)$ be a maximal orbit of the convex foliation
$(\mathcal{E}(t))_{t\in\mathbb{R}}$. Then the rotation angle
$s\mapsto\alpha(s)$ satisfies
\begin{align}\label{rotation-alpha}
   \underset{s\to
\pm\infty}{\rm lim} \alpha(s)=\pm\infty.
\end{align}
\end{lemma}

\noindent See Figure~\ref{dess-traj-osc} for a generic numerical simulation of the maximal orbit of the function $f$
associated with the convex foliation of Figure~\ref{dess:feuilletage-ellipses}.

\begin{proof}
We use the parametrization given by~\eqref{param-s}. Since
$$ \underset{s\to +\infty}{\rm lim} \Vert\gamma (s)\Vert = +\infty\quad\text{and} \quad \underset{t\to -\infty} {\rm lim}\gamma (s)= \mathcal{O},$$ 
we can assume that the function $s\mapsto t(s)$ satisfies
\begin{eqnarray}\label{cond-tprime}
t'(s)>0 \quad \text{and}\quad \underset{s \to \pm \infty}{\rm lim} t(s) = \pm \infty.
\end{eqnarray}
The goal is to compute $\alpha(s)$ using the orthogonality 
condition~\eqref{cond-orth}, which is equivalent to
\begin{eqnarray}\label{cond-ps}
\big\langle \,\gamma' (s),\, \partial_\theta M\!\left(t(s),\theta (s)\right) \big\rangle =0, \quad \text{for all $s\in\mathbb{R}$}.
\end{eqnarray}
Using the notations of Section~\ref{sec:foliation}, we have
\begin{eqnarray*}
\gamma'(s)=\frac{d}{ds} M(t(s),\theta (s))
= t' \partial_t (Rm) + \theta' \partial_{\theta} (Rm)
= t' (R' m + R \partial_t m)+  \theta' R\partial_{\theta} m
\end{eqnarray*}
and $\partial_\theta M= \partial_\theta (Rm)= R \partial_\theta m$.
It follows
\begin{eqnarray*}
\big\langle \gamma' (s), \partial_\theta M \big\rangle
&=&   t' \big\langle R' m, R \partial_\theta m \big\rangle
+  t' \big\langle R \partial_t m, R \partial_\theta m \big\rangle
+ \theta' \big\langle R \partial_\theta m, R \partial_\theta m \big\rangle\\
&=&
t' \big\langle R(\frac{\pi}{2}) m, \partial_\theta m \big\rangle
+  t' \big\langle \partial_t m, \partial_\theta m \big\rangle
+ \theta' \Vert\partial_\theta m\Vert^2\\
&=&
t'\left(ab +(bb'-aa') \cos\theta \sin\theta\right)+\theta'\left(a^2\sin^2\theta + b^2\cos^2\theta\right).
\end{eqnarray*}
By~\eqref{cond-orth}, we have $\big\langle\gamma^{\prime}(s),\partial_{\theta}M\big\rangle=0$ 
and after substitution $a(t)=\sqrt{2}e^t$ and $b(t)=e^t$ we get
$$
t' e^{2t}(\sqrt{2}- \cos\theta \sin\theta) +\theta'  e^{2t}(1+\sin^2\theta )=0
$$
whence we deduce the following relation between $t(s)$ and $\theta(s)$:
\begin{eqnarray}\label{ttheta}
t'(s)= - \frac{1+\sin^2\theta(s)}{\sqrt{2}- \cos\theta(s) \sin\theta(s)}\theta'(s).
\end{eqnarray}
Since for every $\theta\in \mathbb{R}$ we have
\begin{eqnarray*}
0\,<\, \frac{1}{\sqrt{2}+\frac{1}{2}}\,\leq \, \frac{1+\sin^2\theta}{\sqrt{2}- \cos\theta \sin\theta}\,\leq \,\frac{2}{\sqrt{2}-\frac{1}{2}}\,,
\end{eqnarray*}
we get
\begin{eqnarray*}
-\frac{1}{\sqrt{2}+\frac{1}{2}}\,\theta'(s)\,\leq \,  t'(s)\,\leq \,- \frac{2}{\sqrt{2}-\frac{1}{2}}\,\theta'(s).
\end{eqnarray*}
Therefore, from~\eqref{cond-tprime} we deduce
\begin{eqnarray}\label{cond-thetaprime}
\theta'(s)<0, \quad \theta(s)\mathop{\to}_{s\to -\infty}+\infty, \quad \theta(s)\mathop{\to}_{s\to +\infty}-\infty.
\end{eqnarray}
Next, we establish the relation between $\theta(s)$ and $\phi(s)$, see Figure~\ref{dess_ellipse-rot}.
We have
\begin{eqnarray*}
&& \cos\phi =\frac{a\cos\theta}{\sqrt{a^2\cos^2\theta + b^2\sin^2\theta}}
= \frac{\sqrt{2}\cos\theta}{\sqrt{2\cos^2\theta + \sin^2\theta}},\\
&& \sin\phi =\frac{b\sin\theta}{\sqrt{a^2\cos^2\theta + b^2\sin^2\theta}}
= \frac{\sin\theta}{\sqrt{2\cos^2\theta + \sin^2\theta}}.
\end{eqnarray*}
Differentiating $\cos\phi$ and plugging the result in the second expression,
we end up with
\begin{eqnarray}\label{phitheta}
\phi' =\frac{\sqrt{2}}{1+ \cos^2\theta}  \theta'.
\end{eqnarray}
Assembling~\eqref{ttheta} and~\eqref{phitheta}, we obtain
\begin{eqnarray}\label{expr-alpha-prime}
\alpha'=t'+\phi'=\left(\frac{\sqrt{2}}{1+\cos^2\theta}-\frac{1+\sin^2\theta}{\sqrt{2}-\cos\theta\sin\theta}\right)\theta'=: h(\theta) \theta'.
\end{eqnarray}
\begin{figure}[ht]
\begin{center}
\includegraphics[width=12cm]{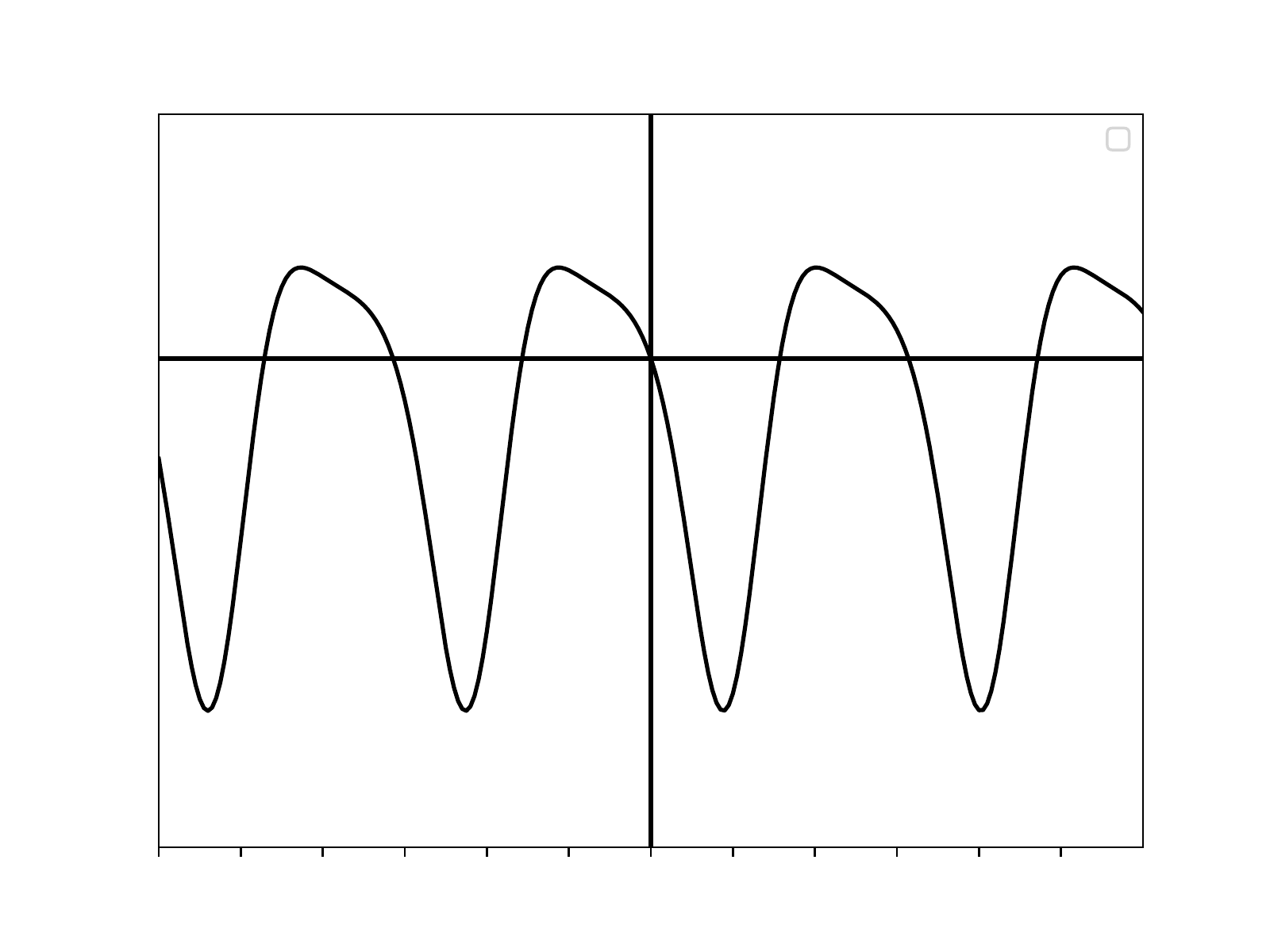}
  \end{center}
\caption{Plot of $\displaystyle h(\theta)= \frac{\sqrt{2}}{1+\cos^2\theta}-\frac{1+\sin^2\theta}{\sqrt{2}-\cos\theta\sin\theta}$.}
\label{dess:fonction-alpha-prime}
\end{figure}
The function $h$ is analytic and $2\pi$-periodic, see Figure~\ref{dess:fonction-alpha-prime}. We can expand
it in Fourier series and integrate~\eqref{expr-alpha-prime} to obtain
\begin{eqnarray}\label{val-alpha}
\alpha (s)= \frac{a_0}{2}\theta (s) + O(1),
\end{eqnarray}
where $O(1)$ is a bounded function and
\begin{eqnarray*}
a_0=\frac{1}{\pi}\int_0^{2\pi}h(\theta)d\theta \simeq -0.84 < 0.
\end{eqnarray*}
We finally conclude from~\eqref{val-alpha} and~\eqref{cond-tprime} that~\eqref{rotation-alpha}
holds.
\end{proof}


\section{Proof of Theorem~\ref{thm-main}}
\label{sec:proof-thm}
Consider the convex foliation by ellipses $\{\mathcal{E}(t)\}_{t\in\mathbb{R}}$ given by Lemma~\ref{lem-feuilletage}. 
Let $k\geq 1$ be any integer and $f$ be the convex function defined by Lemma~\ref{lem-fct-convexe}
for $0<\tau < \min\{ 1/10, 1/k \}$. Then, by Lemma~\ref{lem-regularite-fct}, the function $f$ is coercive, has
its unique minimum at the origin $\mathcal{O}$, is real analytic in  $\mathbb{R}^2\setminus\{\mathcal{O}\}$
and satisfies the {\L}ojasiewicz inequality~\eqref{eq:1}.
Further, Lemma~\ref{lem-Ck-fct}, ensures that $f$ is $\mathcal{C}^k$-smooth.
Finally, Lemma~\ref{lem-traj-oscill} asserts that all nontrivial gradient orbits spiral infinitely many times both near the origin (bounded part) and at infinity.\hfill$\square$

\vspace{0.5cm}

\noindent\textbf{Acknowledgement.} This work was partially supported by the
Centre Henri Lebesgue ANR-11-LABX-0020-01 and the grants CMM AFB170001, 
ECOS-Sud/ANID C18E04 and FONDECYT 1211217. Major part of this work has been
done during a research visit of the first author to INSA Rennes. This author
is indebted to his hosts for hospitality.

\noindent \rule{4cm}{1.5pt} 
\bigskip

\noindent Aris Daniilidis\smallskip

\noindent DIM--CMM, CNRS IRL 2807\newline Beauchef 851, FCFM, Universidad de
Chile \smallskip\\
\noindent E-mail: \texttt{arisd@dim.uchile.cl} \newline
\texttt{http://www.dim.uchile.cl/\symbol{126}arisd/} \medskip

\noindent Research supported by the grants: \smallskip \newline CMM AFB170001, ECOS-ANID
C18E04,  Fondecyt 1211217 (Chile), \\ PGC2018-097960-B-C22 (Spain and EU).

\vspace{0.5cm}

\bigskip

\noindent Mounir Haddou, Olivier Ley \smallskip

\noindent Univ Rennes, INSA, CNRS, IRMAR - UMR 6625, F-35000 Rennes, France
\smallskip \\
\noindent E-mail: \texttt{\{mounir.haddou, olivier.ley\}@insa-rennes.fr}
\newline\noindent\texttt{http://\{haddou, ley\}.perso.math.cnrs.fr/} \smallskip

\noindent Research supported by the Centre Henri Lebesgue ANR-11-LABX-0020-01.

\end{document}